\documentclass[12pt,a4paper,reqno]{amsart}
\usepackage[marginratio=1:1,totalwidth=15.75cm,totalheight=22.275cm]{geometry}
\usepackage[utf8]{inputenc}
\usepackage{amsmath,amssymb,amsthm}

\usepackage[shortlabels]{enumitem}
\usepackage{hyperref}
\usepackage{color}

\newcommand*{\defeq}{\mathrel{\vcenter{\baselineskip0.5ex \lineskiplimit0pt
                     \hbox{\scriptsize.}\hbox{\scriptsize.}}}%
                     =}

\def\Cx{\mathbb{C}}

\def\one{\mathbf{1}}
\newcommand{\ip}[2]{\left\langle #1, #2 \right\rangle}
\newcommand{\norm}[1]{\left\| #1 \right\|}
\newcommand{\fl}[1]{\left\lfloor #1 \right\rfloor}

\newcommand{\DD}{\mathbb{D}}

\theoremstyle{plain}
\newtheorem{thm}{Theorem}[section]
\newtheorem{lemma}[thm]{Lemma}
\newtheorem{cor}[thm]{Corollary}
\newtheorem{prop}[thm]{Proposition}
\newtheorem{thmout}{Theorem}

\theoremstyle{definition}
\newtheorem{defi}[thm]{Definition}

\title{Mixing and ergodicity of compositions of inner functions}
\author{Gustavo R. Ferreira}
\address{Centre de Recerca Matem\`atica, Barcelona, Spain}
\email{grodrigues@crm.cat}
\author{Artur Nicolau}
\address{Departament de Matem\`atiques, Universitat Aut\`onoma de Barcelona, and Centre de Recerca Matem\`atica, 08193, Barcelona, Spain}
\email{artur.nicolau@uab.cat}
\thanks{The first author thanks Phil Rippon and Gwyneth Stallard for many fruitful discussions, and especially for pointing out reference \cite{BS01}. The second author is supported in part by the Generalitat de Catalunya (grant 2021 SGR 00071) and the Spanish Ministerio de Ciencia e Innovaci\'on (project  PID2021-123151NB-I00). Both authors acknowledge financial support from the Spanish Research Agency through the Mar\'ia de Maeztu Program (CEX2020-001084-M)}
\date{\today}

\begin{document}
\begin{abstract}
We study ergodic and mixing properties of non-autonomous dynamics on the unit circle generated by inner functions fixing the origin. 
\end{abstract}
\maketitle

\section{Introduction}\label{sec:intro}
Let $\DD$ be the open unit disc of the complex plane and let $g\colon\DD\to\DD$ be an analytic mapping with $g(0)=0$ which is not a rotation. On one hand, the classical Denjoy-Wolff Theorem tells us that the iterates $g^n = g \circ \ldots \circ g$ converge to $0$ uniformly on compact subsets of $\DD$ (see e.g. \cite[p. 79]{CG93}). On the other hand, let $m$ denote normalised Lebesgue measure on the unit circle $\partial \DD$. An analytic self-mapping $g$ of $\DD$ is called \emph{inner} if
\[ 
\hat g(e^{i\theta}) \defeq \lim_{r\nearrow 1} g(re^{i\theta}) 
\]
exists and has modulus one for $m$-almost every point $e^{i\theta}\in\partial\DD$. In this case, one can investigate the dynamics of the measurable boundary self-map $\hat g\colon\partial\DD\to\partial\DD$, which is actually defined at almost every point of $\partial \DD$. If $g$ is an inner function fixing the origin, Lowner's Lemma tells us that $m$ is invariant under $\hat g$, that is, $m ({\hat g}^{-1} (E)) = m (E)$ for any measurable set $E \subset \partial \DD$. If furthermore $g$ is not a rotation, it is well known that the mapping $\hat g$ is exact and hence, mixing and ergodic. In fact, dynamical properties of the boundary map of an inner function have been extensively studied after the pioneering papers of Aaranson, Pommerenke, Crazier, and Doering and Mañ\'e \cite{Aar78,Pom81,Cra91,DM91}. Mapping and distortion properties of inner functions have been studied in \cite{ref:AleksandrovMeasurablePartitionsCircle,ref:AleksandrovMultiplicityBoundaryValuesInnerFunctions,FP92,FPR96,FMP07} and the surveys \cite{ref:PoltoratskiSarasonACMeasures, ref:SaksmanACMeasures},  and several stochastic properties can be found in \cite{ref:NicolauSolerGibert2022,N22,IU23,AN24}. In many ways, these papers highlight the beautiful interplay between the dynamical properties of an inner function as a self-mapping of $\partial \DD$ and those as a self-mapping of $\DD$. 

The main purpose of this paper is to study ergodic and mixing properties of \textit{non-autonomous dynamics} of inner functions fixing the origin. This concerns compositions of the form $G_n \defeq g_n\circ g_{n-1}\circ\cdots\circ g_1$, where each $g_n$, $n\in\mathbb{N}$, is an inner function with $g_n (0) = 0$. Non-autonomous dynamics of inner functions appear naturally in complex dynamics when studying simply connected wandering domains of entire functions -- see e.g. \cite[Section 2]{BEFRS19} or \cite[Lemma  3.2]{Fer21}. As introduced in \cite{BEFRS22}, a sequence $\{g_n\}_{n\in\mathbb{N}}$ of inner functions fixing the origin is called \emph{contracting} if $G_n\to 0$ uniformly on compact subsets of $\DD$ as $n\to\infty$. The behaviour of $G_n$ in $\DD$ has been studied in \cite[Section 2]{BEFRS19} and \cite{Fer23}, where the following dichotomy has been proved. 

\begin{thmout}[\cite{BEFRS19} and \cite{Fer23}]
\label{thm:limits}
Let $g_n\colon\DD\to\DD$ be inner functions fixing the origin, and let $G_n\defeq g_n\circ\cdots \circ g_1$, $n\in\mathbb{N}$. 
\begin{enumerate}[(a)]
    \item The sequence $\{g_n\}_{n\in\mathbb{N}}$ is contracting if and only if $\sum_{n\geq 1} (1 - |g_n'(0)|) = \infty$.
    \item Assume $\{g_n\}_{n\in\mathbb{N}}$ is not contracting. Then any  (pointwise) limit function in $\DD$ of a  subsequence of $\{G_n\}_{n\in\mathbb{N}}$ is a non-constant inner function fixing the origin, and any two limit functions $H_1$ and $H_2$ satisfy $H_1 =\lambda \cdot H_2$ for some $\lambda \in \partial\DD$.
\end{enumerate}
\end{thmout}

Much less has been said about non-autonomous dynamics of inner functions in the unit circle. In this paper, we focus on ergodic and mixing properties of non-autonomous dynamics of inner functions fixing the origin (for the ``opposite'' case where $G_n(0)$ tends to $\partial\DD$, see \cite{BEFRS22} for recurrence properties of such non-autonomous systems). In the non-autonomous setting, there exist measure-preserving sequences of transformations which are mixing in the usual sense, but such that time averages do not converge to the space average almost everywhere (see e.g. \cite{BS01}). In order to preserve the ``mixing implies ergodicity'' and the ``mixing is equivalent to ergodicity of any subsequence'' maxims and properly generalise the relevant concepts, we use the following definitions due to Berend and Bergelson \cite{BB84}. For $(X, \mathcal{A} , \mu)$ a probability space, we (here and henceforth) denote by $L^2 (\mu)$ the standard Hilbert space of complex valued measurable functions $\varphi$ defined on $X$ such that
\[
\|\varphi \|_2^2 = \int_X |\varphi|^2 d\mu < \infty. 
\]
\begin{defi}\label{def:mixergod}
Let $(X, \mathcal{A} , \mu)$ be a probability space. Let $f_n\colon X\to X$ be measurable, measure-preserving transformations and let $T_n\defeq f_n\circ\cdots \circ f_1$, $n\in\mathbb{N}$. We say that the sequence $\{T_n\}_{n\in\mathbb{N}}$ is \textit{ergodic} if
\[ \lim_{N \to \infty} \norm{\frac{1}{N}\sum_{n=1}^N \varphi\circ T_n - \int_X \varphi\,d\mu}_2 = 0,  \]
for every $\varphi\in L^2(\mu)$. 

The sequence $\{T_n\}_{n\in\mathbb{N}}$ is called \textit{mixing} if all its subsequences are ergodic.
\end{defi}

It is not hard to show that if $\{T_n \}_{n\in\mathbb{N}}$ is ergodic in this sense, then there exist no non-trivial completely invariant sets. Similarly, if $\{T_n \}_{n\in\mathbb{N}}$ is mixing in the sense of the above definition, then it is also mixing in the usual sense, that is, $ \mu\left(A\cap T_n^{-1}(B)\right)\to \mu(A)\mu(B)$ as $n\to+\infty$, for any pair of measurable sets $A, B\subset X$. Notice that the converse statements do not hold, as evidenced by the examples in \cite{BS01}. Moreover, the conditions given in Definition \ref{def:mixergod} properly generalise the usual ones -- more specifically, if $T_n = f^n$ for some measure-preserving map $f\colon X\to X$, then the sequence $\{T_n\}_{n\in\mathbb{N}}$ is ergodic (resp. mixing) if and only if $f$ is ergodic (resp. mixing); see \cite{BB84}.

We now fix some notation. Given inner functions $g_n$ fixing the origin, $n\geq 1$, we consider the inner functions 
\begin{equation}
    \label{notation}
    G_n\defeq g_n\circ\cdots\circ g_1,  \qquad G_m^n \defeq g_n\circ\cdots\circ g_{m+1}, \, n > m \geq 0 .
\end{equation}
Notice that $G_0^n = G_n$ and $G_{n-1}^{n} = g_n$ for any $n\geq 1$. We say that the sequence $\{\widehat G_n\}_{n\in\mathbb{N}}$ is ergodic (respectively mixing) if the corresponding condition in Definition \ref{def:mixergod} holds. We can now state our first result.

\begin{thm}\label{thm:ergodic} Let $g_n$, $n\in\mathbb{N}$,  be inner functions fixing the origin and let  $G_n$, $G_m^n$ be given by \eqref{notation}. The sequence $\{\widehat G_n \}_{n\in\mathbb{N}}$ is ergodic if and only if 
\begin{equation}\label{eq:ergodiccondition}
    \lim_{N\to+\infty} \Re\left(\frac{1}{N^2}\sum_{m=1}^{N-1} \sum_{n = m + 1}^N \left((G_m^n)'(0)\right)^\ell\right) = 0, 
\end{equation}
for any $\ell\in\mathbb{N}$.
\end{thm}
Since $((G_m^n)'(0))^\ell$ are complex numbers, the double sum in Condition (\ref{eq:ergodiccondition}) may have cancellations, so that (\ref{eq:ergodiccondition}) does not depend solely on the modulus of $(G_m^n)'(0)$. With that in mind, in Section \ref{sec:further} we give a necessary and a sufficient condition for ergodicity. Using these conditions, in the extreme cases when $g_n' (0) > 0$ for every $n\in\mathbb{N}$ or when $\{g_n\}_{n\in\mathbb{N}}$ is not contracting, we obtain the following descriptions. 

\begin{thm}\label{thm:easieriff}
Let $g_n$, $n\in\mathbb{N}$,  be inner functions fixing the origin and let $G_n$, $G_m^n$ be given by \eqref{notation}.
\begin{enumerate}[(i)]
    \item Assume $g_n'(0) > 0$ for all $n \geq 1$. Then, the sequence $\{\widehat G_n \}_{n\in\mathbb{N}}$ is ergodic if and only if
    \begin{equation}\label{eq:mod}
    \lim_{N \to \infty } \prod_{k=\fl{N(1 - \varepsilon)}}^N g_k'(0) = 0, 
    \end{equation}
     for any $0< \varepsilon < 1$.
\item Assume $g_n'(0) > 0$ for all $n \geq 1$. Then, the sequence $\{\widehat G_n \}_{n\in\mathbb{N}}$ is mixing if and only if for every $\epsilon > 0$ there exists $M_0$ such that, for any $N > M > M_0$, we have
    \begin{equation}\label{eq:mod1}
    \prod_{k = N - M }^N g_k'(0) < \epsilon. 
    \end{equation}
    \item Assume the sequence $\{g_n\}_{n\in\mathbb{N}}$ is not contracting. Then the sequence $\{ \widehat G_n\}_{n\in\mathbb{N}}$ is ergodic if and only if the sequence $\{e^{i\arg G_n'(0)}\}_{n\in\mathbb{N}}$ is equidistributed on $\partial\DD$. 
\end{enumerate}
\end{thm}
Theorem \ref{thm:easieriff} deserves some comments.

First, when $g_n(z) = e^{i\theta}z$, $n \geq 1$, the statement (iii) tells us that $\{ \widehat G_n\}_{n\in\mathbb{N}}$ is ergodic if and only if $\theta$ is irrational -- in other words, we recover the classical result of ergodic theory that a rotation of the circle is ergodic if and only if it is irrational. Condition \eqref{eq:mod} in (i) is related to the speed of convergence of $G_n (z)$, $z \in \DD$, to $0$. Roughly speaking, the faster $G_n$ tends to $0$ on $ \DD$, the more expanding $\widehat G_n$ is on $\partial \DD$; see \cite[Theorem 4.8]{Mas13}. Hence, condition (i) is related to the classical existence of ergodic measures for expanding maps; for related ideas and results, see e.g. \cite[Chapter 11]{VO16}, \cite{GS09}, and even \cite{TPvS19} for a non-autonomous approach. Thus, roughly speaking, Theorem \ref{thm:easieriff} says that in our setting there are only two mechanisms for ergodicity, and they are the classical ones corresponding to ``expanding maps'' and ``irrational rotations''. 

Second, Pommerenke \cite{Pom81} showed that contracting sequences of inner functions fixing the origin are mixing in the usual sense on $\partial \DD$, that is, $ m \left(A\cap G_n^{-1}(B)\right)\to m(A) m(B)$ as $n\to+\infty$, for any pair of measurable sets $A, B \subset \partial \DD$. We discuss a converse to Pommerenke's result in Section \ref{sec:usual}.

Third, note that, by Theorem \ref{thm:limits}, $\{g_n \}_{n\in\mathbb{N}}$ is contracting if and only if 
\[
\prod_{j=k}^\infty g_j' (0)=0 \text{ for all $k\in\mathbb{N}$.}
\]
Hence conditions \eqref{eq:mod} and \eqref{eq:mod1} can be understood as quantitative versions of contractibility. In particular, condition \eqref{eq:mod} outlined in Theorem \ref{thm:easieriff}(i) implies that the sequence $\{g_n \}_{n\in\mathbb{N}}$ is contracting, but is strictly stronger; in Section \ref{sec:examples}, we give an example of a contracting sequence $\{g_n \}_{n\in\mathbb{N}}$ which does not satisfy (\ref{eq:ergodiccondition}), and is therefore not ergodic in the sense of Definition \ref{def:mixergod}. Notice that, by the results of Pommerenke \cite{Pom81} mentioned above, this sequence is also mixing in the usual sense; we see that (as in \cite{BS01}) a sequence can be mixing in the usual sense, but not ergodic. From Theorem \ref{thm:easieriff} we derive the following more straightforward sufficient conditions for ergodicity and mixing.
\begin{cor}
\label{cor:mix} 
Let $g_n$, $n\in\mathbb{N}$,  be inner functions fixing the origin and let $G_n$ be given by \eqref{notation}. 
\begin{enumerate}[(a)]
\item Assume that  for any $0 < \varepsilon < 1$ we have 
\begin{equation}
    \label{suferg}
   \lim_{N \to \infty} \prod_{k=\fl{N(1 - \varepsilon)}}^N g_k'(0) = 0.
\end{equation}
Then the sequence $\{\widehat G_n \}_{n\in\mathbb{N}}$ is ergodic.
\item Assume that for every $\epsilon > 0$ there exists $M_0$ such that, for $N > M > M_0$,
\begin{equation}
    \label{sufmix}
\prod_{k=N-M}^N |g_k'(0)| < \epsilon.
\end{equation}
Then the sequence $\{\widehat G_n \}_{n\in\mathbb{N}}$ is mixing.
\end{enumerate}
\end{cor}

Combining this result and part (iii) of Theorem \ref{thm:easieriff} we obtain that contracting sequences are precisely those which have a mixing subsequence. Notice the similarity to Theorem \ref{thm:mixusual}.

 \begin{cor}
\label{cor:mix1}
Let $g_n$, $n\in\mathbb{N}$,  be inner functions fixing the origin and let $G_n$ be given by \eqref{notation}. Then $\{g_n \}_{n\in \mathbb{N}}$ is contracting if and only if
$\{\widehat G_n \}_{n\in\mathbb{N}}$ has a mixing subsequence.
\end{cor}

Finally, we wish to discuss what our results mean for the \textit{recurrence} of the sequence $\{\widehat G_n \}_{n\in\mathbb{N}}$. Let $(X,\mathcal{A} , \mu)$ be a measure space. A sequence of measurable, measure-preserving transformations $T_n\colon X \to X$ is recurrent if for any measurable set $A \subset X$ we have that $T_n(x)\in A$ for infinitely many $n \geq 1$, for $\mu$-almost every point $x\in A$. In autonomous dynamics, that is when $T_n = T^n$, a classical theorem of Poincar\'e (see e.g. \cite[Theorem 1.2.1]{VO16}) tells us that if $\mu (X) < \infty$, then the sequence $\{T^n \}$ is recurrent for any measure-preserving transformation $T$. However, this may fail if the preserved measure is infinite, or if the system is non-autonomous; see \cite{DM91} and \cite{BEFRS22}. It turns out, however, that ergodicity in the sense of Definition \ref{def:mixergod} is sufficient for recurrence. In fact, the following stronger result, which is well-known in the autonomous case, holds. 
\begin{thm}\label{prop:rec}
Let $(X, \mathcal{A}, \mu)$ be a probability space. Let $T_n\colon X\to X$, $n \geq 1$, be measurable, measure-preserving transformations. If $\{T_n\}_{n\in\mathbb{N}}$ has an ergodic subsequence, then:
\begin{enumerate}[(i)]
    \item $\{T_n\}_{n\in\mathbb{N}}$ is recurrent;
    \item If, furthermore, $\mathrm{supp}(\mu) = X$ and $X$ is a second countable topological space, then $\mu$-almost every point has a dense orbit in $X$.
\end{enumerate}
\end{thm}

For other results on recurrence and topological transitivity of boundary extensions of inner functions and compositions thereof, see \cite{DM91} and \cite{BEFRS22}.

The paper is organized as follows. In Section \ref{sec:prelims} we collect some standard definitions and classical results which will be used later. Section \ref{sec:general} is devoted to the proof of Theorem \ref{thm:ergodic}. In Section \ref{sec:further} we present one necessary and one sufficient condition for ergodicity, and use them to prove Theorem \ref{thm:easieriff} and Corollary \ref{cor:mix}. Section \ref{sec:usual} is devoted to prove the converse of Pommerenke's result on mixing sequences of inner functions. Section \ref{sec:examples} contains some relevant examples and the proof of Corollary \ref{cor:mix1}. Finally, Theorem \ref{prop:rec} is proved in Section \ref{sec:rec}.

\section{Preliminaries}
\label{sec:prelims}
\subsection{Inner functions and the space $L^2 (\partial \DD)$ }
\label{ssec:L2}
Let $L^2(\partial \DD)$ be the usual Hilbert space of measurable complex-valued functions $\varphi\colon\partial\DD\to\Cx$ such that
\[ \norm{\varphi}_2 \defeq \left(\int_{\partial\DD} |\varphi|^2\,dm \right)^{1/2} < +\infty, \]
armed with the corresponding inner product
\[ \ip{\varphi}{\psi} \defeq \int_{\partial\DD} \varphi\cdot\overline{\psi}\,d m = \int_0^{2 \pi} \varphi (e^{i\theta}) \cdot\overline{\psi (e^{i\theta})}\,\frac{d \theta}{2 \pi }. 
\]
Finite linear combinations of the trigonometric monomials
\[ e_n(e^{i\theta}) \defeq e^{i n \theta},\quad n\in\mathbb{Z}, \]
are dense in $L^2(\partial\DD)$. The boundary extension $\hat g$ of an inner function $g\colon\DD\to\DD$ satisfies $|\hat g (\xi)| = 1$ for almost every $\xi \in \partial \DD$. Cauchy's Formula tells us that the Fourier coefficients of $\hat g$ are precisely the coefficients of the power series expansion of $g$ at the origin, that is,
\begin{equation}\label{eq:riesz}
    \ip{\hat g}{e_n} = \int_0^{2 \pi} \hat g(e^{i\theta}) e^{-ni\theta}\, \frac{d \theta}{2 \pi} = \begin{cases}
        0, & n < 0, \\
        \frac{g^{(n)}(0)}{n!}, & n \geq 0.
    \end{cases}
\end{equation}

Let $\omega_z$ denote the harmonic measure on $\partial\DD$ with respect to the point $z\in\DD$, defined as 
\[
\omega_z (E) = \int_E \frac{1-|z|^2}{|\xi - z|^2} dm (\xi) , \quad E  \subset \partial \DD . 
\]
We recall the following classical result (see \cite[Corollary 1.5]{DM91}):
\begin{lemma}[Lowner's Lemma]
\label{lem:innerharmonic}
Let $g\colon\DD\to\DD$ be an inner function. Then for any $z\in\DD$,
\[  \omega_{g(z)} (E) = \omega_z (g^{-1}(E)), \quad E \subset \partial \DD . \]
In particular, if $g(0)=0$ and $\varphi$ is an integrable function on $\partial \DD$, we have
\[
\int_{\partial \DD} (\varphi \circ g ) dm = \int_{\partial \DD} \varphi dm . 
\]
\end{lemma}

\subsection{Equidistribution of sequences on $\partial\DD$}
This is a much-studied notion with many connections. We recall the basic definition:
\begin{defi}
We say that a sequence $\{z_n\}_{n\in\mathbb{N}}\subset\partial\DD$ is equidistributed on $\partial\DD$ if for any arc $S\subset\partial\DD$, we have
\[ \lim_{n\to+\infty} \frac{\#(\{z_1, z_2, \ldots, z_n\}\cap S)}{n} = m (S). \]
\end{defi}

The following classical characterisation, due to Weyl (see, for instance, \cite[Theorem 2.1]{KN74}) will allow us to discuss equidistribution in our setting. 
\begin{lemma}[Weyl's Criterion]
A sequence $\{z_n\}_{n\in\mathbb{N}}\subset \partial\DD$ is equidistributed in $\partial\DD$ if and only if
\[ \lim_{N\to+\infty} \frac{1}{N}\sum_{n=1}^N (z_n)^{\ell} = 0,   \]
for any integer $\ell \geq 1$.
\end{lemma}

\section{A general characterisation of ergodicity}\label{sec:general}
In this section, we prove Theorem \ref{thm:ergodic}. We start with the following elementary lemma.
\begin{lemma}\label{lem:derivative}
Let $g\colon\DD\to\DD$ be holomorphic with $g(0) = 0$. Then, for all $n\in\mathbb{N}$,
\[ \left(g^n\right)^{(n)}(0) = n!\left(g'(0)\right)^n. \]
\end{lemma}
\begin{proof}
For $n>0$, we use the notation $O(|z|^n)$ to denote a function $h$ defined on $\DD$ for which there exists a constant $C=C(h,n)>0$  such that $|h(z)| \leq C |z|^n$ for any $|z|< 1/2$.  Since $g(z) = g' (0) z + O(|z|^2)$, we have $g(z)^n = g'(0)^n z^n + O(|z|^{n+1}) $ and the result follows by uniqueness of the Taylor series.

\end{proof}

Next we show that suitable inner products on the circle can be written as derivatives at the origin. As before, for $ \ell \in \mathbb{Z} $, let $e_\ell$ denote the monomial $e_\ell (\xi) = {\xi}^\ell$, $\xi \in \partial \DD$.
\begin{lemma}\label{lem:cauchy}
Let $g_n$, $n \geq 1$,  be inner functions fixing the origin and $G_n$, $G_m^n$ be given by \eqref{notation}. Then, for any integers $n>m \geq 0$ and $\ell\in\mathbb{N}$,
\[ \ip{e_\ell\circ \widehat G_n}{e_\ell\circ \widehat G_m} = \left((G_m^n)'(0)\right)^\ell. \]
\end{lemma}
\begin{proof}
Fix the integers $n>m \geq 0$ and $\ell\in\mathbb{N}$. By Lemma \ref{lem:innerharmonic}, we have
\[ \ip{e_\ell\circ \widehat G_n}{e_\ell\circ \widehat G_m} = \ip{e_\ell\circ \widehat G_m^n}{e_\ell}. \]
The right-hand side can be read as the $\ell$-th Fourier coefficient of the boundary map of the holomorphic self-map of $\DD$ given by $z\mapsto \left(G_m^n(z)\right)^\ell$, and so by (\ref{eq:riesz}) we get
\[ \ip{e_\ell\circ \widehat G_n}{e_\ell\circ \widehat G_m} = \frac{\left((G_m^n)^\ell\right)^{(\ell)}(0)}{\ell!}. \]
The conclusion follows by applying Lemma \ref{lem:derivative} to the right-hand side.
\end{proof}

We are ready to prove Theorem \ref{thm:ergodic}.
\begin{proof}[Proof of Theorem \ref{thm:ergodic}]
Since trigonometric polynomials are dense in $L^2 (\partial \DD)$, it is sufficient to show that
\[
\lim_{N \to \infty} \norm{\frac{1}{N}\sum_{n=1}^N e_\ell\circ \widehat G_n}_2^2 = 0, \quad \ell \in \mathbb{Z} \setminus \{0\} . 
\]
Since $e_{ - \ell} = \overline{e_\ell}$, we can assume $\ell \geq 1$. Fix $\ell \geq 1$. We have
\[ \norm{\frac{1}{N}\sum_{n=1}^N e_\ell\circ \widehat G_n}_2^2 =
\frac{1}{N^2}\sum_{n=1}^N \norm{e_\ell\circ \widehat G_n}_2^2 + 2\Re\left(\frac{1}{N^2}\sum_{m=1}^{N-1}\sum_{n=m+1}^N \ip{e_\ell\circ \widehat G_n}{e_\ell\circ \widehat G_m}\right).
\]
Since $|e_\ell\circ \widehat G_n| = 1$ $m$-almost everywhere on $\partial \DD$ for any $n \geq 1$, we obtain 
\[ \norm{\frac{1}{N}\sum_{n=1}^N e_\ell\circ \widehat G_n}_2^2 = \frac{1}{N} + 2\Re\left(\frac{1}{N^2}\sum_{m=1}^{N-1} \sum_{n=m+1}^N \ip{e_\ell\circ \widehat G_n}{e_\ell\circ \widehat G_m}\right). \]
Applying Lemma \ref{lem:cauchy}, we have $\ip{e_\ell\circ \widehat G_n}{e_\ell\circ \widehat G_m} = \left((G_m^n)'(0)\right)^\ell$, which finishes the proof.
\end{proof}

\section{Ergodicity in different scenarios}\label{sec:further}
The proof of Theorem \ref{thm:easieriff} is based on the following result.
\begin{thm}\label{thm:necsuf}
With the same notation as Theorem \ref{thm:ergodic}, the following hold:
\begin{enumerate}[(a)]
    \item The sequence $\{ \widehat G_n \}_{n\in\mathbb{N}}$ is ergodic if, for all $\ell\in\mathbb{N}$,
    \begin{equation}
    \label{eq:suf}
    \lim_{N\to+\infty} \frac{1}{N}\sum_{m=1}^{N-1} \left((G_m^N)'(0)\right)^\ell = 0.
    \end{equation}

    \item If the sequence $\{ \widehat G_n \}_{n\in\mathbb{N}}$ is ergodic, then
\begin{equation}
    \label{eq:nec}
    \lim_{N\to+\infty} \frac{1}{N} \sum_{n=m+1}^N \left((G_m^n)'(0)\right)^\ell = 0,  
    \end{equation}
    for every pair of integers $m \geq 0$ and $\ell \geq 1$.
\end{enumerate}
\end{thm}
\begin{proof}[Proof of Theorem \ref{thm:necsuf}]
We will prove (a) using Theorem \ref{thm:ergodic}. We start by noting that
\[ \Re\left(\frac{1}{N^2}\sum_{m=1}^{N-1} \sum_{n = m + 1}^N \left((G_m^n)'(0)\right)^\ell\right) \leq \left|\frac{1}{N^2}\sum_{m=1}^{N-1} \sum_{n = m + 1}^N \left((G_m^n)'(0)\right)^\ell\right|. \]
We rewrite the sum on the right-hand side as
\[ \frac{1}{N^2}\sum_{m=1}^{N-1}\sum_{n=m+1}^N \left((G_m^n)'(0)\right)^\ell = \frac{1}{N^2}\sum_{n=2}^N \sum_{m=1}^{n-1} \left((G_m^n)'(0)\right)^\ell, \]
and applying the triangle inequality yields
\[ \left|\frac{1}{N^2}\sum_{m=1}^N\sum_{n=m+1}^N \left((G_m^n)'(0)\right)^\ell\right| \leq \frac{1}{N^2}\sum_{n=2}^N\left|\sum_{m=1}^{n-1} \left((G_m^n)'(0)\right)^\ell\right| \leq \frac{1}{N}\sum_{n=2}^N\left|\frac{1}{n}\sum_{m=1}^{n-1} \left((G_m^n)'(0)\right)^\ell\right|. \]
The right-hand side is now the Ces\`aro sum of a sequence that, by hypothesis, goes to zero. Therefore condition (\ref{eq:ergodiccondition}) of Theorem \ref{thm:ergodic} is satisfied. This completes the proof of (a).

We now prove (b), which is actually independent of Theorem \ref{thm:ergodic}. Fix integers $m \geq 0$ and $\ell \geq 1$. We have by the Cauchy--Schwarz inequality that
\[ \ip{\frac{1}{N}\sum_{n=1}^N e_\ell\circ \widehat G_n}{e_\ell\circ \widehat G_m} \leq \left\|\frac{1}{N}\sum_{n=1}^N e_\ell\circ \widehat G_n\right\|_2\left\|e_\ell\circ \widehat G_m\right\|_2, \]
and since $|e_\ell\circ \widehat G_m | = 1$ at $m$-almost every point of the unit circle, this becomes
\begin{equation}\label{eq:ineq}
    \ip{\frac{1}{N}\sum_{n=1}^N e_\ell\circ \widehat G_n}{e_\ell\circ \widehat G_m} \leq \left\|\frac{1}{N}\sum_{n=1}^N e_\ell\circ \widehat G_n\right\|_2.
\end{equation}
For $0 \leq m < N$, we have
\[
\ip{\frac{1}{N}\sum_{n=1}^N e_\ell\circ \widehat G_n}{e_\ell\circ \widehat G_m} 
 = \frac{1}{N}\left(\sum_{n=1}^m \ip{e_\ell\circ \widehat G_n}{e_\ell\circ \widehat G_m} + \sum_{n=m+1}^N \ip{e_\ell\circ \widehat G_n}{e_\ell\circ \widehat G_m}\right).
\]
Applying Lemma \ref{lem:cauchy} and plugging everything back into (\ref{eq:ineq}), we obtain
\[ \frac{1}{N}\sum_{n=1}^m \ip{e_\ell\circ \widehat G_n}{e_\ell\circ \widehat G_m} + \frac{1}{N}\sum_{n=m+1}^N \left((G_m^n)'(0)\right)^\ell \leq \left\|\frac{1}{N}\sum_{n=1}^N e_\ell\circ \widehat G_n\right\|_2. \]
If we assume that $\{\widehat G_n\}_{n\in\mathbb{N}}$ is ergodic, the right-hand side now goes to zero as $N\to+\infty$; the first sum on the left-hand side depends only on $m$ and $\ell$, and thus when divided by $N$ goes to zero as $N\to+\infty$. It follows that
\[ \lim_{N\to+\infty} \frac{1}{N}\sum_{n=m+1}^N\left((G_m^n)'(0)\right)^\ell = 0. \]
\end{proof}

We will also need the following observation about Ces\`aro sums.
\begin{lemma}\label{lem:cessums}
Let $\{z_n\}_{n\in\mathbb{N}}$ and $\{w_n\}_{n\in\mathbb{N}}$ be two sequences of complex numbers. Assume that $\sup_n |z_n| < \infty$ and that $\lim_{n \to \infty} w_n = w\in \mathbb{C} \setminus \{0\}$. Then,
\[ \lim_{N\to+\infty} \frac{1}{N}\sum_{n=1}^N z_nw_n = w\cdot\lim_{N\to+\infty} \frac{1}{N}\sum_{n=1}^N z_n. \]
The equality is strong in the sense that either both limits exist and the identity holds, or neither limit exists.
\end{lemma}
\begin{proof}
    The proof easily follows from the observation
    \begin{equation*}
        \lim_{N\to+\infty} \frac{1}{N}\sum_{n=1}^N z_n (w_n -w) = 0.
    \end{equation*}
\end{proof}
We are ready to prove Theorem \ref{thm:easieriff}.
\begin{proof}[Proof of Theorem \ref{thm:easieriff}]
We assume now that $g_n'(0) > 0$ for all $n \geq 1$. To show that (\ref{eq:mod}) implies ergodicity, we will show that it implies the sufficient condition \eqref{eq:suf} given in Theorem \ref{thm:necsuf}(a). First, notice that since now $(G_m^N)'(0) \in (0, 1)$, it suffices to show that the condition \eqref{eq:suf} is satisfied for $\ell = 1$. Fixed $\varepsilon > 0$, we decompose the sum in question as
\[ \frac{1}{N}\sum_{m=1}^{N-1} (G_m^N)'(0) = \frac{1}{N}\sum_{m=1}^{\fl{N(1-\varepsilon)}} \prod_{k=m+1}^N g_k'(0) + \frac{1}{N}\sum_{m=\fl{N(1-\varepsilon)}+1}^{N-1} (G_m^N)'(0). \]
For the first sum on the right-hand side, note that, since $m \leq \fl{N(1 - \varepsilon)}$ and $g_n'(0) \in (0, 1)$, we have
\[ \prod_{k=m+1}^N g_k'(0) \leq \prod_{k=\fl{N(1 - \varepsilon)}+1}^N g_k'(0). \]
For the second sum, note that each $(G_m^N)'(0)$ is less than one, and the sum itself has at most $N - 1 - (\fl{N(1 - \varepsilon)} + 1) \leq \varepsilon N$ terms. Thus,
\[ \frac{1}{N}\sum_{m=1}^{N-1} (G_m^N)'(0) \leq (1 - \varepsilon)\prod_{k=\fl{N(1 -\varepsilon)}+1}^N g_k'(0) +\varepsilon. \]
By assumption, the first term on the right-hand side goes to zero as $N\to+\infty$, and since $\varepsilon > 0$ was arbitrary it follows that
\[ \lim_{N\to+\infty} \frac{1}{N} \sum_{m=1}^{N-1} (G_m^N)'(0) = 0. \]
Ergodicity follows by Theorem \ref{thm:necsuf}(a). To show the necessity of condition (\ref{eq:mod}), we use the characterisation of ergodicity given in Theorem \ref{thm:ergodic}. Assume that there exist constants $\varepsilon > 0$, $c>0$ and a subsequence $\{N_k \}_{k\in\mathbb{N}}$ of positive integers such that
\[ \prod_{j=\fl{N_k(1 - \epsilon)}+1}^{N_k} g_j'(0) \geq c > 0, \quad k\in\mathbb{N}. \]
We will show that condition (\ref{eq:ergodiccondition}) of Theorem \ref{thm:ergodic} fails for $\ell = 1$. Since all the derivatives are positive and $(G_m^n)'(0) \geq (G_m^{N_k})'(0)$ for any $0\leq m < n \leq N_k$, we get
\[ \frac{1}{N_k^2}\sum_{m=1}^{N_k-1}\sum_{n=m+1}^{N_k} (G_m^n)'(0) \geq \frac{1}{N_k^2}\sum_{m=1}^{N_k-1}\sum_{n=m+1}^{N_k} (G_m^{N_k})'(0) = \frac{1}{N_k^2}\sum_{m=1}^{N_k-1} (N_k - m)(G_m^{N_k})'(0)  . \]
Since all terms in the sum are positive, we have
\[ \frac{1}{N_k^2}\sum_{m=1}^{N_k-1}\sum_{n=m+1}^{N_k} (G_m^n)'(0)  \geq \frac{1}{N_k^2}\sum_{m=\fl{N_k(1-\varepsilon)}+1}^{N_k-1} (N_k - m)(G_m^{N_k})'(0). \]
Now, for $ N_k > m > \fl{N_k(1-\epsilon)}$, we once again have 
\[ (G_m^{N_k})'(0) \geq (G_{\fl{N_k(1-\varepsilon)}}^{N_k})'(0) = \prod_{j=\fl{N_k(1 - \varepsilon)}+1}^{N_k} g_j'(0) \geq c > 0,  \]
and therefore
\[ \frac{1}{N_k^2}\sum_{m=1}^{N_k-1}\sum_{n=m+1}^{N_k} (G_m^n)'(0) \geq \frac{c}{N_k^2}\sum_{m=\fl{N_k(1-\varepsilon)}+1}^{N_k-1} (N_k - m). \]
The sum on the right-hand side is of order $N_k^2$. We conclude that
\[ \liminf_{k\to+\infty} \frac{1}{N_k^2}\sum_{m=1}^{N_k-1}\sum_{n=m+1}^{N_k} (G_m^n)'(0) > 0, \]
meaning that condition (\ref{eq:ergodiccondition}) fails. This completes the proof of (i).

To prove (ii), note that $\{ \widehat G_n \}_{n\in\mathbb{N}}$ is mixing if and only if $\{ \widehat G_{n_k} \}_{k\in\mathbb{N}}$ is ergodic for any subsequence of positive integers $\{n_k\}_{k\in\mathbb{N}}$. Since $G_{n_k}$ corresponds to the non-autonomous dynamics of the inner functions $\tilde{g_k} = g_{n_k} \circ \ldots \circ g_{n_{k-1} + 1}$, part (i) yields that $\{ \widehat G_n \}_{n\in\mathbb{N}}$ is mixing if and only if for any $0 < \varepsilon < 1$, we have  
\[
\prod_{k=\fl{ N(1- \varepsilon)}}^N \prod_{j= n_{k-1} + 1}^{n_k} g_j' (0) \to 0,
\]
as $N \to \infty$, for any subsequence of positive integers $\{n_k\}_{k\in\mathbb{N}}$. This last statement is equivalent to \eqref{eq:mod1}. 

To prove (iii), we assume that $\sum (1- |g_n' (0)|) < \infty$ (which, by Theorem \ref{thm:limits}, is equivalent to the sequence $\{g_n\}_{n\in\mathbb{N}}$ not being contracting). In this case, we can assume (by discarding finitely many $g_n$) that $g_n'(0) \neq 0$ for all $n\in\mathbb{N}$. Since the sequence $ \{|G_n'(0)| \}_{n\in\mathbb{N}}$ is decreasing and by assumption is bounded away from $0$, it has a positive limit $c>0$ as $n\to+\infty$. We now assume that $\{\widehat G_n \}_{n\in\mathbb{N}}$ is ergodic and invoke Theorem \ref{thm:necsuf}(b) with $m = 0$, obtaining
\[ \lim_{N\to+\infty} \frac{1}{N}\sum_{n=1}^N (G_n'(0))^\ell = 0 \]
for every $\ell\in\mathbb{N}$. Lemma \ref{lem:cessums} now says that
\[ 0 = \lim_{N\to+\infty} \frac{1}{N}\sum_{n=1}^N |G_n'(0)|^\ell (e^{i\arg G_n'(0)})^\ell = c^{\ell} \cdot\lim_{N\to+\infty} \frac{1}{N}\sum_{n=1}^N e^{i \ell \arg G_n'(0)} , \quad \ell\in\mathbb{N}, \]
and it follows from Weyl's criterion that the sequence $\{e^{i\arg G_n'(0)}\}_{n\in\mathbb{N}}$ is equidistributed in $\partial\DD$.

Conversely, assume that $\{ e^{i\arg G_n'(0)} \}_{n\in\mathbb{N}}$ is equidistributed in $\partial\DD$. Using the chain rule, we rewrite $(G_m^n)'(0)$ as $(G_m^n)'(0) = G_n'(0)\cdot (G_m'(0))^{-1}$, so that the sum in equation \eqref{eq:suf} of Theorem \ref{thm:necsuf}(a) becomes
\[ \frac{1}{N} \sum_{m=1}^N ((G_m^N)'(0))^{\ell} =   \frac{G_N'(0)^{\ell}}{N} \sum_{m=1}^N (G_m'(0))^{-\ell} = \frac{G_N'(0)^{\ell}}{N} \sum_{m=1}^N \frac{1}{|G_m'(0)|^\ell}e^{-i\ell\arg G_m'(0)} , \quad \ell \in \mathbb{N} . \]
By Lemma \ref{lem:cessums}, we get 
\[ \lim_{N\to+\infty} \frac{1}{N}\sum_{m=1}^N ((G_m^N)'(0))^\ell = \lim_{N\to+\infty} \frac{1}{N}\sum_{m=1}^N e^{-i\ell\arg G_m'(0)}, \quad \ell\in\mathbb{N} \]
Since $\{e^{i\arg G_n'(0)}\}_{n\in\mathbb{N}}$ is equidistributed on $\partial\DD$, then by Weyl's criterion this limit is zero, whence ergodicity follows by Theorem \ref{thm:necsuf}(a).
\end{proof}

\begin{proof}[Proof of Corollary \ref{cor:mix}]

The proof mimics the previous argument and we only sketch it. For part (a) we need to show that \eqref{suferg} implies the sufficient condition \eqref{eq:suf} given in Theorem \ref{thm:necsuf}(a). This follows as in the proof of part (a) of Theorem \ref{thm:easieriff} once triangular inequality is applied. Part (b) follows similarly.
\end{proof}

\section{Mixing in the usual sense}\label{sec:usual}
Recall that a sequence $\{T_n\}_{n\in\mathbb{N}}$ of transformations of the measure space $(X, \mathcal{A}, \mu)$ is \textit{mixing} (in the usual sense) if, for all measurable sets $A, B\subset X$,
\[ \mu\left(A\cap T_n^{-1}(B)\right)\to \mu(A)\mu(B) \text{ as $n\to+\infty$.} \]
If $T_n = T^n$ and $\mu$ is finite, this implies that $\{T^n \}$ is ergodic (see e.g. \cite[Proposition 4.1.3]{VO16}). However, if the system is non-autonomous, this implication can fail drastically -- see e.g. \cite{BS01}. Nevertheless, it can be interesting (and useful; see \cite[Theorem 7.4]{BEFRS22}) to study ``classical mixing'' for compositions of inner functions.

In this vein, as previously mentioned, Pommerenke \cite{Pom81} already showed that, if $g_n\colon\DD\to\DD$ are inner functions fixing the origin and the composition $G_n\defeq g_n\circ\cdots\circ g_1$ tends to zero locally uniformly in $\DD$ as $n\to+\infty$, then $\{\widehat G_n\}_{n\in\mathbb{N}}$ is mixing in the usual sense (in fact, Pommerenke showed the stronger fact that $\{\widehat G_n\}_{n\in\mathbb{N}}$ is ``exact in the usual sense''). Here, we give a converse to his result for non-autonomous dynamics. The proof relies on the following consequence of mixing (see \cite[Corollary 7.1.14]{VO16}), whose short proof is included for completeness.
\begin{lemma}\label{lem:acmixing}
Let $(X, \mathcal{A} , \mu)$ be a probability space. Let $f_n\colon X \to X$ be measurable, measure-preserving transformations. Assume that the sequence $F_n\defeq f_n\circ \cdots\circ f_1$ is mixing in the usual sense. Let $\nu$ be a probability measure on $X$ which is absolutely continuous with respect to $\mu$. Then,
\[ \lim_{n \to \infty } \nu\left(F_n^{-1}(B)\right) = \mu(B) \]
for any measurable set $B\subset X$.
\end{lemma}
\begin{proof}
Let $\varphi$ denote the Radon-Nykodim derivative of $\nu$ relative to $\mu$, and let $\one_B$ denote the indicator function of the measurable set $B \subset X$. Since $\varphi$ can be approximated in $L^1 (\mu)$ by linear combinations of characteristic functions, the assumption that $\{F_n\}_{n\in\mathbb{N}}$ is mixing gives that 
\[ \int_X (\one_B\circ F_n)\cdot\varphi\,d\mu\to \int_X \one_B\,d\mu\int_X \varphi\,d\mu  \]
as $n\to+\infty$. The left-hand side is, by the Radon-Nykodim Theorem, equal to $\nu(F_n^{-1}(B))$, while the right-hand side is equal to $\mu(B)$ since $\nu$ is a probability measure.
\end{proof}

\begin{thm}\label{thm:mixusual}
Let $g_n\colon\DD\to\DD$ be inner functions fixing the origin, and let $G_n\defeq g_n\circ\cdots\circ g_1$, $n \geq 1$. Then, the sequence $\{\widehat G_n\}_{n\in\mathbb{N}}$ is mixing in the usual sense if and only if $\{g_n\}_{n\in\mathbb{N}}$ is contracting.
\end{thm}

\begin{proof}[Proof of Theorem \ref{thm:mixusual}]
As mentioned before, Pommerenke (\cite{Pom81}) proved that $\{ \widehat G_n\}_{n\in\mathbb{N}}$ is mixing if $G_n$ tend to $0$ uniformly on compacts of $\DD$. Conversely, assume that $\{\widehat G_n \}_{n\in\mathbb{N}}$ is mixing but at the same time $G_n\to G$ pointwise in $\DD$, where $G$ is a non-constant inner function. Now, take $z\in\DD\setminus G^{-1}(0)$. Since the harmonic measure $\omega_z$ is absolutely continuous with respect to Lebesgue measure, Lemma \ref{lem:acmixing} gives that
\[ \lim_{n \to \infty } \omega_z\left(\widehat G_n^{-1}(B)\right) =  m(B) \]
for every measurable set $B\subset\partial\DD$. However, by Lemma \ref{lem:innerharmonic}, we have 
\[ \omega_z\left(\widehat G_n^{-1}(B)\right) = \omega_{G_n(z)}(B), \]
for any measurable set $B\subset\partial\DD$. Since $G_n(z)\to G(z)$ we have $\omega_z\left(\widehat G_n^{-1}(B)\right)\to \omega_{G(z)}(B)$ as $n \to \infty$, for any measurable set $B \subset \partial \DD$. Since $G(z) \neq 0$, there exists a measurable set $B \subset \partial \DD$ with $w_{G(z)} (B) \neq m(B)$ and we obtain a contradiction, concluding the proof.
\end{proof}

\section{Examples and counterexamples}\label{sec:examples}
In this section, we apply the various necessary and sufficient conditions obtained above to illustrate what ergodic and non-ergodic compositions of inner functions may look like. We start with the example promised in Section \ref{sec:intro} of a sequence that is contracting but not ergodic. This example also serves to show that the necessary condition given in Theorem \ref{thm:necsuf}(b) cannot be sufficient. 
\begin{prop}    
\label{thm:counterexample}
There exists a sequence $g_n\colon\DD\to\DD$ of inner functions fixing the origin such that the sequence $\{G_n\}_{n\in\mathbb{N}}$ generated by $G_n\defeq g_n\circ\cdots\circ g_1$ satisfies the following conditions:
\begin{enumerate}
    \item $G_n\to 0$ locally uniformly in $\DD$;
    \item $\{G_n\}_{n\in\mathbb{N}}$ satisfies the necessary condition \eqref{eq:nec} in Theorem \ref{thm:necsuf}(b);
    \item $\{\widehat G_n\}_{n\in\mathbb{N}}$ is not ergodic.
\end{enumerate}
\end{prop}
\begin{proof}
Let $g_n\colon\DD\to\DD$ be the Blaschke product of degree $2$ given by
\[ g_n(z) = z\cdot\frac{z + a_n}{1 + a_nz}, \]
with $a_n = n/(n+1)$. An immediate calculation shows that $g_n'(0) = a_n$, and so
\[ \sum_{n\geq 1} (1 - |g_n'(0)|) = \sum_{n\geq 1} \frac{1}{n+1} = \infty, \]
whence $G_n = g_n\circ\cdots \circ g_1$ converges locally uniformly to zero by Theorem \ref{thm:limits}. Furthermore, by the chain rule, we have
\begin{equation}
\label{eq:deriv}
    (G_m^n)'(0) = \prod_{k=m+1}^n \frac{k}{k+1} = \frac{m+1}{n+1},
\end{equation}
and so for any fixed natural numbers $\ell$ and $m$, the sequence $((G_m^n)'(0))^\ell$ goes to zero as $n\to+\infty$. The necessary condition \eqref{eq:nec} in Theorem \ref{thm:necsuf}(b) is now satisfied, since it becomes the Ces\`aro sum of a sequence going to zero. We finally show that $\{ \widehat G_n\}$ is not ergodic. Note that $m+1 \geq (n+1) /2$ if $N/2 \leq m < n \leq N$. Hence \eqref{eq:deriv} gives 
\[
\frac{1}{N^2} \sum_{m=1}^{N-1} \sum_{n=m+1}^N  (G_m^n)'(0) \geq \frac{1}{2N^2} \sum_{m=\fl{N/2}}^{N-1} (N-m),
\]
which does not tend to $0$ as $N \to \infty$. Hence condition \eqref{eq:ergodiccondition} for $\ell = 1$ in Theorem \ref{thm:ergodic} is not satisfied and consequently $\{ \widehat G_n\}_{n\in\mathbb{N}}$ is not ergodic.  
\end{proof}

Next, we use Theorem \ref{thm:easieriff} to provide several explicit examples of mixing and ergodic compositions of inner functions.
\begin{cor}\label{cor:examples}
Let $g_n\colon\DD\to\DD$ be inner functions fixing the origin and let $G_n \defeq g_n\circ\cdots\circ g_1$, $n \geq 1$. 
\begin{enumerate}[(i)]
\item Assume $\sum (1- |g_n' (0)|) = \infty$. Then $\{\widehat G_n\}_{n\in\mathbb{N}}$ has a mixing subsequence.

\item If, furthermore, there exist constants $0< \lambda < 1$, $0<c<1$ and $M_0 >0$ such that for any $a,b \in \mathbb{N}$ with $b-a >M_0$ one has 
\begin{equation}
    \label{eq:card}
\# \{n \in [a,b] : |g_n' (0)| \leq \lambda \} \geq c (b-a) , 
\end{equation}
then the sequence $\{\widehat G_n \}_{n\in\mathbb{N}}$ is mixing.
    \item Assume $\sum (1- |g_n' (0)|) < \infty$. Then:
    \begin{enumerate}[(a)]
        \item If $\arg g_n'(0)\to\theta$ as $n\to+\infty$ for some $\theta\in\mathbb{R}\setminus\mathbb{Q}$, then the sequence $\{\widehat G_n\}_{n\in\mathbb{N}}$ is ergodic.
        \item If the arguments $\theta_n = \arg g_n'(0)$ are independently and identically distributed according to some non-atomic distribution on $\partial\DD$, then $\{\widehat G_n\}_{n\in\mathbb{N}}$ is ergodic with probability 1.
    \end{enumerate}
\end{enumerate}
\end{cor}
\begin{proof}


Assume $\sum (1 - |g_n' (0)|) = \infty$. Then there exists an increasing sequence $\{N_k \}_{k\in\mathbb{N}}$ of positive integers such that 
\[
\prod_{j= N_k +1 }^{N_{k+1}} |g_j'(0)| \leq 1/2, \quad k=1,2,\ldots . 
\]
Since
\[
\prod_{k=M}^N \prod_{j=N_k +1 }^{N_{k+1}} |g_j'(0)| \leq 1/2^{N-M}, 
\]
part (b) of Corollary \ref{cor:mix} gives that the subsequence $\{\widehat G_{N_k}\}_{k\in\mathbb{N}}$ is mixing.  

Assume now that condition \eqref{eq:card} holds. Then 
\[
\prod_{j=M}^N |g_j' (0)| \leq \lambda^{c(N-M)}
\]
if $N-M \geq M_0$, whence by part (b) of Corollary \ref{cor:mix} the sequence $\{\widehat G_n \}_{n\in\mathbb{N}}$ is mixing.

Now, assume that $G_n\not\to 0$ locally uniformly on $\DD$ (which, recall, is equivalent to $\sum_{n\geq 1}(1 - |g_n'(0)|)<\infty$ by Theorem \ref{thm:limits}), whence by Theorem \ref{thm:easieriff}(iii) the sequence $\{\widehat G_n\}_{n\in\mathbb{N}}$ is ergodic if and only if $\{e^{i\arg G_n'(0)}\}_{n\in\mathbb{N}}$ is equidistributed on $\partial\DD$. Thus, we only need to check that the conditions outlined in (a) and (b) imply equidistribution. That (a) does is an immediate consequence of a theorem by van der Corput (see \cite[Theorem 3.3]{KN74}). On the other hand, (b) implies equidistribution almost surely by a result of Robbins \cite[Theorem 2]{Rob53}, which says that sums of independent and identically distributed random variables drawn from a non-atomic distribution are equidistributed with probability one.
\end{proof}

We can now prove Corollary \ref{cor:mix1}.

\begin{proof}[Proof of Corollary \ref{cor:mix1}]
That contracting implies having a mixing subsequence is part (a) of the previous result, and the converse follows from part (iii) of Theorem \ref{thm:easieriff}.
\end{proof}

\section{Recurrence}\label{sec:rec}
In this section, we prove Theorem \ref{prop:rec}.
\begin{proof}[Proof of Theorem \ref{prop:rec}]
We can assume without loss of generality that the sequence $\{T_n\}_{n\in\mathbb{N}}$ is ergodic.

To show recurrence, let $A\subset X$ be a measurable set with $\mu(A) > 0$. If $\one_A$ is the characteristic function of $A$, then by ergodicity the functions
\[ x\mapsto \frac{1}{N}\sum_{n=1}^N \one_A\circ T_n(x) \]
converge to the constant function $\mu(A)$ in $L^2(\mu)$. Since a sequence converging in $L^2(\mu)$ admits a subsequence converging $\mu$-almost everywhere, recurrence follows.

Next, assume that $\mathrm{supp}(\mu) = X$, and let $\{U_k\}_{k\in\mathbb{N}}$ be a countable basis for the topology of $X$. For $k \in \mathbb{N}$, denote by $\tilde U_k$ the set of points $x\in X$ such that $T_n(x)\in U_k$ for only finitely many $n \geq 1$. Then, clearly,
\[ \frac{1}{N} \sum_{n=1}^N \one_{U_k}\circ T_n(x) \to 0\text{ as $N\to+\infty$} \]
for all $x\in \tilde U_k$. Hence, no subsequence of the time averages at $x \in \tilde U_k$ converge to $\mu(U_k) > 0$, and so (by ergodicity) $\tilde U_k$ must have measure zero. The set
\[ \tilde U = \bigcup_{k\in\mathbb{N}} \tilde U_k \]
also has measure zero, and it is clear that any point $x\in X\setminus \tilde U$ has a dense orbit.
\end{proof}



\begin{thebibliography}{BEF{\etalchar{+}}24}

\bibitem[Aar78]{Aar78}
J.~Aaronson, \emph{Ergodic theory for inner functions of the upper half plane}, Ann. Inst. H. Poincar{\'e} Sect. B \textbf{14} (1978), 233--253.

\bibitem[Ale86]{ref:AleksandrovMeasurablePartitionsCircle}
A.~B. Aleksandrov, \emph{Measurable partitions of the circle induced by inner functions}, Zap. Nauchn. Sem. Leningrad. Otdel. Mat. Inst. Steklov. (LOMI) \textbf{149} (1986), no.~Issled. Line\u{i}n. Teor. Funktsi\u{i}. XV, 103--106, 188. \MR{849298}

\bibitem[Ale87]{ref:AleksandrovMultiplicityBoundaryValuesInnerFunctions}
\bysame, \emph{Multiplicity of boundary values of inner functions}, Izv. Akad. Nauk Armyan. SSR Ser. Mat. \textbf{22} (1987), no.~5, 490--503, 515. \MR{931885}

\bibitem[AN23]{AN24}
J.~Aaronson and M.~Nadkarni, \emph{Dynamics of inner functions revisited}, 2023, available at \href{https://arxiv.org/abs/2308.16063}{\texttt{arXiv:2305.02042}}.

\bibitem[BB84]{BB84}
D.~Berend and V.~Bergelson, \emph{Ergodic and mixing sequences of transformations}, Ergod. Th. Dyn. Sys. \textbf{4} (1984), 353--366.

\bibitem[BEF{\etalchar{+}}22]{BEFRS19}
A.~M. Benini, V.~Evdoridou, N.~Fagella, P.~J. Rippon, and G.~M. Stallard, \emph{Classifying simply connected wandering domains}, Math. Ann. \textbf{383} (2022), 1127--1178.

\bibitem[BEF{\etalchar{+}}24]{BEFRS22}
\bysame, \emph{Boundary dynamics for holomorphic sequences, non-autonomous dynamical systems and wandering domains}, 2024, to appear in \textit{Adv. Math.} Available at \href{https://arxiv.org/abs/2203.06235}{\texttt{arXiv:2203.06235v3}}.

\bibitem[BS01]{BS01}
E.~Behrends and J.~Schmeling, \emph{Strongly mixing sequences of measure-preserving transformations}, Czechoslov. Math. J. \textbf{51} (2001), 377--385.

\bibitem[CG93]{CG93}
L.~Carleson and T.~W. Gamelin, \emph{Complex dynamics}, Springer, 1993.

\bibitem[Cra91]{Cra91}
M.~Craizer, \emph{Entropy of inner functions}, Israel J. Math. \textbf{74} (1991), 129--168.

\bibitem[DM91]{DM91}
C.~I. Doering and R.~Ma{\~n\'e}, \emph{The dynamics of inner functions}, Ensaios Matem{\'a}ticos, vol.~3, Sociedade Brasileira de Matem{\'a}tica, 1991.

\bibitem[Fer23]{Fer23}
G.~R. Ferreira, \emph{A note on forward iteration of inner functions}, Bull. London Math. Soc. \textbf{55} (2023), 1143--1153.

\bibitem[Fer24]{Fer21}
\bysame, \emph{Multiply connected wandering domains of meromorphic functions: the pursuit of uniform internal dynamics}, Ergod. Th. Dyn. Sys. \textbf{44} (2024), 727--748.

\bibitem[FMP07]{FMP07}
J.L. Fernández, M.V. Melián, and D.~Pestana, \emph{Quantitative mixing results and inner functions}, Math. Ann. \textbf{337, no.1} (2007), 233--251.

\bibitem[FP92]{FP92}
J.L. Fernández and D.~Pestana, \emph{Distortion of boundary sets under inner functions and applications}, Indiana Univ.Math.J. \textbf{41, no.2} (1992), 439--448.

\bibitem[FPR96]{FPR96}
J.L. Fernández, D.~Pestana, and J.M. Rodr\'iguez, \emph{Distortion of boundary sets under inner functions. ii}, Pacific J. Math. \textbf{172} (1996), 49--81.

\bibitem[GS09]{GS09}
J.~Graczyk and S.~Smirnov, \emph{Non-uniform hyperbolicity in complex dynamics}, Invent. Math. \textbf{175} (2009), 335--415.

\bibitem[IU23]{IU23}
O.~Ivrii and M.~Urbanski, \emph{Inner functions, composition operators, symbolic dynamics and thermodynamic formalism}, 2023, available at \href{https://arxiv.org/abs/2308.16063}{\texttt{arXiv:2308.16063}}.

\bibitem[KN74]{KN74}
L.~Kuipers and H.~Niederreiter, \emph{Uniform distribution of sequences}, John Wiley {\&} Sons, 1974.

\bibitem[Mas13]{Mas13}
J.~Mashregi, \emph{Derivatives of inner functions}, Fields Institute Monographs, vol.~31, Springer, 2013.

\bibitem[Nic22]{N22}
A.~Nicolau, \emph{Convergence of linear combinations of iterates of an inner function}, J. Math. Pures Appl. \textbf{161, no.9} (2022), 135--165.

\bibitem[NSiG22]{ref:NicolauSolerGibert2022}
Artur Nicolau and Od\'{\i} Soler~i Gibert, \emph{A central limit theorem for inner functions}, Adv. Math. \textbf{401} (2022), Paper No. 108318, 39. \MR{4394685}

\bibitem[Pom81]{Pom81}
Ch. Pommerenke, \emph{On ergodic properties of inner functions}, Math. Ann. \textbf{256} (1981), 43--50.

\bibitem[PS06]{ref:PoltoratskiSarasonACMeasures}
Alexei Poltoratski and Donald Sarason, \emph{Aleksandrov-{C}lark measures}, Recent advances in operator-related function theory, Contemp. Math., vol. 393, Amer. Math. Soc., Providence, RI, 2006, pp.~1--14. \MR{2198367}

\bibitem[Rob53]{Rob53}
H.~Robbins, \emph{On the equidistribution of sums of independent random variables}, Proc. Amer. Math. Soc. \textbf{4} (1953), 786--799.

\bibitem[Sak07]{ref:SaksmanACMeasures}
Eero Saksman, \emph{An elementary introduction to {C}lark measures}, Topics in complex analysis and operator theory, Univ. M\'{a}laga, M\'{a}laga, 2007, pp.~85--136. \MR{2394657}

\bibitem[TPvS19]{TPvS19}
M.~Tanzi, T.~Pereira, and S.~van Strien, \emph{Robustness of ergodic properties of non-autonomous piecewise expanding maps}, Ergod. Th. Dynam. Sys. \textbf{39} (2019), 1121--1152.

\bibitem[VO16]{VO16}
M.~Viana and K.~Oliveira, \emph{Foundations of ergodic theory}, Cambridge University Press, 2016.

\end{thebibliography}
\newcommand{\etalchar}[1]{$^{#1}$}
\providecommand{\bysame}{\leavevmode\hbox to3em{\hrulefill}\thinspace}
\providecommand{\MR}{\relax\ifhmode\unskip\space\fi MR }
\providecommand{\MRhref}[2]{%
  \href{http://www.ams.org/mathscinet-getitem?mr=#1}{#2}
}
\providecommand{\href}[2]{#2}

\end{document}